\documentclass[12pt]{article}
\usepackage{amscd}
\usepackage[all]{xy}
\usepackage{bbm}
\usepackage{amssymb}
\usepackage{latexsym}
\usepackage{amsmath}
\usepackage{amsthm}

\newtheorem{Theorem}{Theorem}[section]
\newtheorem{Lemma}[Theorem]{Lemma}

\newtheorem{Proposition}[Theorem]{Proposition}
\newtheorem{Remark}[Theorem]{Remark}

\begin{document}
\begin{center}
{\large\bf Small rational curves on the moduli space of stable bundles\ }

\vspace{0.2cm}

Liu Min \\
School of Mathematical Sciences,\\
 Fudan University, Shanghai 200433, P.R.China\\
E-mail: liumin@amss.ac.cn\\
\end{center}
\footnotetext{ Supported in part by the National Natural Science Foundation of China (Grant No. 10871106) }

\begin{abstract}
For a smooth projective curve $C$ with genus $g\geq 2$ and a degree 1 line bundle
$\mathcal{L}$ on $C$, let $M:=SU_{C}(r,\mathcal{L})$ be the moduli space of stable vector bundles of rank $r$ over $C$ with the fixed determinant $\mathcal{L}$. In this paper, we study the small rational curves on $M$ and estimate the codimension of the locus of the small rational curves. In particular, we determine all small rational curves when $r=3$.
\end{abstract}



\section{ Introduction }	
Let $C$ be a smooth projective curve of genus $g\geq 2$ and
$\mathcal{L}$ a line bundle on $C$ of degree $d$. Let
$M:=SU_{C}(r,\mathcal{L})$ be the moduli space of stable vector
bundles of rank $r$ and with the fixed determinant $\mathcal{L}$.
Throughout this paper, we assume that $d=1$. In this case, $M$ is a
smooth projective Fano variety with $Pic(M)=\mathbb{Z}\cdot\Theta$
and $-K_{M}=2\Theta$, where $\Theta$ is an ample divisor (\cite{Ramanan}, \cite{DrezetNarasimhan}). So, as a Fano
variety, the \emph{index} of $M$ is $2$.

For any smooth Fano variety $X$, a rational curve $l\subset X$ is called
a \emph{line on} $X$ if the index of $X$ equals $-K_{X}\cdot l$.
Also we say that $l$ \emph{has degree} $k$ if $-K_{X}\cdot l$ equals
$k$ times the index of $X$. For the moduli space $M$, so defined degree of
a rational curve $l\subset M$ equals $\text{deg}(\Theta|_{l})$. So,
for any rational curve $\phi: \mathbb{P}^{1}\longrightarrow M$, we
can define its degree $\text{deg}(\phi^{*}\Theta)$ with respect to the
ample generator of $\text{Pic}(M)$. It's an open problem if every
smooth Fano variety with picard number 1 has a line ( see Problem
V1.13 of \cite{Kollar} ).

Ramanan \cite{Ramanan} found a family of lines on $M$, but these lines lie
in a proper closed subset. It has been observed that $M$ is covered
by rational curves of degree $r$, which are called \emph{Hecke
curves} (Corollary 5.16 of \cite{NarasimhanRamanan78}). Sun proved
that there are no lines on $M$ except those found by Ramanan
(Theorem 2 of \cite{Sun05}), and he also proved that rational curves of
minimal degree passing through a general point of $M$ are Hecke
curves (Theorem 1 of \cite{Sun05}), which answers a question of Jun-Muk
Hwang (Question 1 of \cite{Hwang01}) and also implies that all the
rational curves of $\Theta-$degree smaller than $r$, called \emph{small rational curves} (\cite{MokSun09}), must lie in a
proper closed subset. It has been proved that the locus of small
rational curves consists of $r-1$ irreducible
components (\cite{Choe11}).

The aim of this paper is to estimate the codimension of the locus of small rational curves and to determine the small rational curves when $r=3$. In particular, we can see that each small rational curve is defined by a vector bundle on $C\times \mathbb{P}^{1}$ of certain fixed extension types. We also construct all small rational curves on $M$ when $r=3$, which will be useful to study the Chow-group of the moduli space $M$.

This paper is organized as follows. In section 2, using the degree formula of rational curves due to Sun (\cite{Sun05}), we show that, for any small rational curve, there exist a sequence fixed semi-stable bundles $V_{1}, \cdots, V_{n-1}$ of degree 0 and a fixed stable bundle $V_{n}$ of degree 1 on $C$ such that the bundles corresponding to points of the small rational curve are obtained by extensions of $V_{1}, \cdots, V_{n-1}$ and $V_{n}$. And we also prove that the locus of small rational curves is a closed subvariety of codimension at least $\text{min}_{0<r_{1}<\cdots< r_{n}=r}\{\sum_{i=2}^{n}r_{i-1}(r_{i}-r_{i-1})(g-2)+r_{n-1}(r_{n}-r_{n-1}-1)\}$, where $0<r_{1}<\cdots< r_{n}=r$ runs over all positive integers $r_{i}$ satisfying $\sum_{i=1}^{n-1}r_{i}(\alpha_{i}-\alpha_{i+1})<r$ for some $n$ and some integers $\alpha_{1}>\cdots>\alpha_{n}$.

In section 3, we determine all small rational curves on $M=SU_{C}(3, \mathcal{L})$. Let $J_{C}$ be the Jacobian of $C$, and $U_{C}(2,1)$ the moduli space of stable bundles on $C$ of rank 2 and degree 1. Let $\mathcal{R}\subset J_{C}\times U_{C}(2,1)$ be the closed subvariety consisting of $([\xi], [V_{2}])$ satisfying $\xi\otimes\text{det}(V_{2})\cong\mathcal{L}$. We construct a projective bundle $q: \mathcal{P}\rightarrow \mathcal{R}$ such that, for any $([\xi], [V_{2}])\in \mathcal{R}$, $q^{-1}(([\xi], [V_{2}]))\cong\mathbb{P}\text{Ext}^{1}(V_{2}, \xi)$. Let $SU_{C}(2,0)$ be the open subset of stable bundles in $U_{C}(2,0)$, and $J_{C}^{1}$ the moduli space of degree 1 line bundles on $C$. Let $\mathcal{R}'\subset SU_{C}(2,0)\times J_{C}^{1}$ be the closed subvariety consisting of $([V_{1}], [\xi'])$ satisfying $\text{det}(V_{1})\otimes \xi'\cong \mathcal{L}$.
By using Hecke transformation, we construct a scheme $\mathbb{P}(\mathcal{V^{\vee}})^{s}$ parametrizing a family of stable bundles of rank 2 and degree 0, which maps onto $SU_{C}(2,0)$. Let $T=(\mathbb{P}(\mathcal{V^{\vee}})^{s}\times J_{C}^{1})\times_{SU_{C}(2,0)\times J_{C}^{1}} \mathcal{R}'$. Then there is a surjective morphism $\theta': T\rightarrow \mathcal{R}'$. We also construct a projective bundle $q': \mathcal{P}'\rightarrow T$ such that, for any $([V_{1}], [\xi'])\in \mathcal{R}'$ and any $t\in \theta'^{-1}([V_{1}], [\xi'])$, $q'^{-1}(t)\cong \mathbb{P}\text{Ext}^{1}(\xi', V_{1})$. We have:
\vspace{0.1cm}

 \begin{Theorem}
 There are morphisms $\Phi: \mathcal{P}\rightarrow M$ and $\Psi: \mathcal{P}'\rightarrow M$ such that any rational curve obtained by the following ways are small rational curves and any small rational curve $\phi: \mathbb{P}^{1}\longrightarrow M$ can be obtained by one of the following ways:

(i) it's the image (under $\Phi$) of a rational curve of degree 2 in $\mathcal{P}$ in the fiber of $q$;

(ii) it's the image (under $\Phi$) of a double cover of a line in $\mathcal{P}$ in the fiber of $q$;

(iii) it's the image (under $\Phi$) of a line in $\mathcal{P}$ which is not in a fiber of $q$ and maps to a line  in $U_{C}(2,\mathcal{L'})$ for some line bundle $\mathcal{L'}$ of degree 1;

(iv) it's the image (under $\Psi$) of a line in $\mathcal{P}'$ in the fiber of $q'$.
\end{Theorem}

\section{The Locus of Small Rational Curves on Moduli Space}
In \cite{Sun05}, it was shown that any rational curve $\phi: \mathbb{P}^{1}\longrightarrow M$ is defined by a vector bundle $E$ on $C\times \mathbb{P}^{1}$ and a degree formula was discovered, which we call Sun's degree formula. To recall it, let $X=C\times \mathbb{P}^{1}$, $f:X\longrightarrow C$ and $\pi:X\longrightarrow \mathbb{P}^{1}$ be the projections. On a general fiber $f^{-1}(\xi)=X_{\xi}$, $E$ has the form
$$E\mid_{X_{\xi}}=\bigoplus_{i=1}^{n}\mathcal{O}_{X_{\xi}}(\alpha_{i})^{\oplus r_{i}}, \quad \alpha_{1}>\dots>\alpha_{n}.$$
The n-tuple $\alpha=(\alpha_{1}^{\oplus r_{1}},\dots,\alpha_{n}^{\oplus r_{n}})$ is called the generic splitting type of $E$. Tensoring $E$ by $\pi^{*}\mathcal{O}_{\mathbb{P}^{1}}(-\alpha_{n})$, we can assume without loss of generality that $\alpha_{n}=0$. Any such $E$ admits a relative Harder-Narasimhan filtration
$$0=E_{0}\subset E_{1}\subset\dots\subset E_{n}=E$$
in which the quotient sheaves $F_{i}=E_{i}/E_{i-1}$ are torsion-free with generic splitting type $(\alpha_{i}^{\oplus r_{i}})$ respectively. Let $F'_{i}=F_{i}\bigotimes \pi^{*}\mathcal{O}_{\mathbb{P}^{1}}(-\alpha_{i}) (i=1,\dots,n)$, thus they have generic splitting type $(0^{\oplus r_{i}})$ respectively. Without risk of confusion, we denote the degree of $F_{i}\  (\text{resp. } E_{i})$ on the general fiber of $\pi$ by $\text{deg}(F_{i})\ (\text{resp. }\text{deg}(E_{i}))$. Accordingly, $\mu(E_{i})\ (\text{resp. }\mu (E))$ denotes the slope of restriction of $E_{i}\ (\text{resp. } E)$ to the general fiber of $\pi$ respectively. Let $\text{rk}(E_{i})$ denote the rank of $E_{i}$. Then we have the following so called Sun's degree formula ( see the formula (2.2) in \cite{Sun05})
\begin{equation}
\text{deg}(\phi^{*}(\Theta))
=r\sum_{i=1}^{n}c_{2}(F'_{i})+\sum_{i=1}^{n-1}(d\text{rk}(E_{i})-r\text{deg}(E_{i}))(\alpha_{i}-\alpha_{i+1}).
\label{eq:2.1}
\end{equation}
When  $\phi: \mathbb{P}^{1}\longrightarrow M$ is a small rational curve, $\text{deg}(\phi^{*}(\Theta))<r$, we have $n\geq 2,\  c_{2}(F'_{i})=0$ and $F'_{i}=f^{*}{V_{i}}$ for some locally free sheaf $V_{i}$ on $C$ by Lemma 2.2 in \cite{Sun05}. So $F_{i}=f^{*}V_{i}\bigotimes \pi^{*}\mathcal{O}_{\mathbb{P}^{1}}(\alpha_{i})$ and $E$ can be obtained by a sequence of extensions
  $$0\longrightarrow f^{*}V_{1}\bigotimes\pi^{*}\mathcal{O}_{\mathbb{P}^{1}}(\alpha_{1})\longrightarrow E_{2}\longrightarrow f^{*}V_{2}\bigotimes\pi^{*}\mathcal{O}_{\mathbb{P}^{1}}(\alpha_{2})\longrightarrow 0$$
  $$0\longrightarrow E_{2}\longrightarrow E_{3}\longrightarrow f^{*}V_{3}\bigotimes\pi^{*}\mathcal{O}_{\mathbb{P}^{1}}(\alpha_{3})\longrightarrow 0$$
  $$\vdots$$
  $$0\longrightarrow E_{n-1}\longrightarrow E_{n}=E\longrightarrow f^{*}V_{n}\longrightarrow 0.$$
  When $d=1$, we have the following elementary lemma:

\begin{Lemma}
  (1) $E_{i}$ are families of semi-stable bundles of degree 0 on $C$ parametrised by $\mathbb{P}^{1}$ for all $1\leq i\leq n-1$.

 (2) $V_{i}$ are semi-stable bundles of degree 0 on $C$ for all $1\leq i\leq n-1$, and $V_{n}$ is stable of degree 1.
  \label{lm:claims}
\end{Lemma}

\begin{proof}
   (1) Since $E$ is a family of stable bundles of slope $\mu (E)=\frac{1}{r}$, all the proper subbundles $E_{i}\ (1\leq i\leq n-1)$ have degrees $\text{deg}(E_{i})\leq 0$. On the other hand, by the degree formula, we have $\text{deg}(E_{i})\geq 0$ for all $1\leq i\leq n-1$. Otherwise, there is an $E_{i_{0}}$ of degree $\text{deg}(E_{i_{0}})\leq -1$ for some $1\leq i_{0}\leq n-1$, and then $\text{deg}(\phi^{*}(\Theta))\geq(\text{rk}(E_{i_{0}})-r\text{deg}(E_{i_{0}}))(\alpha_{i_{0}}-\alpha_{i_{0}+1})\geq 1+r >r$, which contradicts to $\text{deg}(\phi^{*}(\Theta))<r$. Hence $\text{deg}(E_{i})=0$ for all $1\leq i\leq n-1$.

  For any proper subbundle $F$ of the restriction of $E_{i}$ to a fiber of $\pi$ , we consider $F$ as a subbundle of the restriction of $E$ to this fiber, and then we have $\mu(F)<\mu(E)=\frac{1}{r}$ by the stability of the family $E$, so $\text{deg}(F)\leq 0$ and $\mu(F)\leq0=\mu (E_{i})$. This implies that $E_{i}$ is a family of semi-stable bundles parametrised by $\mathbb{P}^{1}$.

  (2) It's easy to prove that $V_{i}$ ($1\leq i\leq n-1$) is of degree 0 and that $V_{n}$ is of degree 1.

  For $i=1$, since $f^{*}V_{1}\bigotimes \pi^{*}\mathcal{O}_{\mathbb{P}^{1}}(\alpha_{1})=E_{1}$ is a family of semi-stable bundles, $V_{1}$ is semi-stable.

  For any proper quotient sheaf $Q$ of $V_{i}$ ($2\leq i\leq n-1 $), we consider $Q$ as a quotient sheaf of the restriction of $E_{i}$ to a general fiber, and then we have $\mu(Q)\geq \mu(E_{i})=0=\mu(V_{i})$, Hence, $V_{i}$ ($2\leq i\leq n-1 $)  is semi-stable of degree 0.

   For any proper quotient sheaf $Q$ of $V_{n}$, we consider $Q$ as a proper quotient sheaf of the restriction of $E$ to a general fiber, and we have $\mu(Q)>\mu(E)=\frac{1}{r}$, so $\text{deg}(Q)\geq1$ and $\mu(Q)\geq\frac{1}{\text{rk}(Q)}\geq\frac{1}{\text{rk}(V_{n})}=\mu(V_{n})$, which implies that $V_{n}$ is semi-stable and hence stable.
   \end{proof}

 When $d=1$, Sun's degree formula (\ref{eq:2.1}) for a small rational curve is
   \begin{equation}
   \text{deg}(\phi^{*}(\Theta))=\sum_{i=1}^{n-1}\text{rk}(E_{i})(\alpha_{i}-\alpha_{i+1}).
   \label{eq:1}
   \end{equation}

 Now, let's estimate the dimension and codimension of the locus of small rational curves.
Let $n\leq r $ be a positive integer, and fix $n$ positive integers $r_{1}<\cdots<r_{n-1}<r_{n}=r$ such that
\begin{equation}
 \sum_{i=1}^{n-1}r_{i}(\alpha_{i}-\alpha_{i+1})<r
 \label{eq:2}
 \end{equation}
for some integers $\alpha_{1}>\dots>\alpha_{n-1}>\alpha_{n}$.

Let $d_{1}=\cdots=d_{n-1}=0$ and $d_{n}=1$. We define a subset $S_{r_{1},\cdots, r_{n}}$ of $M=SU_{C}(r,\mathcal{L})$ as follows:
 \begin{equation}
 S_{r_{1},\cdots, r_{n}}:=
 \left\{[W]\in M \Bigg|
      \begin{aligned} W \text{ has a filtration }0=W_{0}\subset W_{1}\subset\cdots \subset W_{n}=W  \\
       \text{such that } W_{i}/W_{i-1}
      \text{ is a vector }\text{bundle of rank } \\
      r_{i}-r_{i-1} \text{ and degree } d_{i}-d_{i-1} \text{ for any }1\leq i\leq n
      \end{aligned}
 \right\}.
 \end{equation}

From above description, we know that all small rational curves lie in
 $$S=\bigcup_{0<r_{1}<\cdots<r_{n}=r} S_{r_{1},\cdots, r_{n}},$$
 where $r_{1}<\cdots<r_{n}$ runs over $n$ positive integers $r_{i}$ that satisfying inequality (\ref{eq:2}) for some $n$ and some integers $\alpha_{1}>\dots>\alpha_{n}$.

   \begin{Proposition} (i) $S_{r_{1},\cdots, r_{n}}$ is Zariski-closed in $M$, and therefore can be regarded as a reduced subscheme.

     (ii) There is a scheme $\Sigma$ such that $C\times \Sigma$ carries a rank $r$ vector bundle $\mathbb{E}$ and a filtration
           $$0\subset \mathbb{E}_{1}\subset \cdots \subset \mathbb{E}_{n-1}\subset \mathbb{E}_{n}=\mathbb{E}$$
     such that  $\mathbb{E}_{i}/\mathbb{E}_{i-1}$ is a vector bundle of rank $r_{i}-r_{i-1}$, and for any $\sigma\in \Sigma$, the restriction of $\mathbb{E}_{i}/\mathbb{E}_{i-1}$ to $C\times \{\sigma\}$ has degree $d_{i}-d_{i-1}$.
     \end{Proposition}

\begin{proof}
By the construction of the moduli space, $M$ can be regarded as a good quotient of a subscheme $R$ of a quotient scheme by a reductive group $G$. There is a vector bundle $\mathbb{E}'$ on $C\times R$ whose restriction to $C\times \{q\}$ is the bundle represented by the image of $q$ in $M$. $\mathbb{E}'$ extends to a coherent sheaf (denoted by the same symbol) on the closure $C\times \overline{R}$.

Now we will define $\{Q_{i} \text{ and } \mathbb{E}'_{i} \text{ on } C\times Q_{i}\}_{n\geq i\geq 1}$ by descent induction.

Set $Q_{n}= \overline{R}$ and $\mathbb{E}'_{n}=\mathbb{E}'$.

For any  $n\geq i$ $(i>1)$, we assume that $Q_{i} \text{ and } \mathbb{E}'_{i} \text{ on } C\times Q_{i}$ have been constructed. Let $Q_{i-1}$ be the closed subscheme of relative quotient scheme $\text{Quot}_{C\times Q_{i}/Q_{i}}(\mathbb{E}'_{i})$ parametrizing quotients which restrict to vector bundles of rank $r_{i}-r_{i-1}$ and degree $d_{i}-d_{i-1}$ on each $C\times \{q\}$. It's known that there is a universal quotient bundle on $C\times Q_{i}\times_{Q_{i}}Q_{i-1}=C\times Q_{i-1}$
$$\mathbb{E}'_{i}\bigotimes \mathcal{O}_{C\times Q_{i-1}}\longrightarrow \mathcal{V}_{i}\longrightarrow 0.$$
Let $\mathbb{E}'_{i-1}=\text{Ker}(\mathbb{E}'_{i}\bigotimes \mathcal{O}_{C\times Q_{i-1}}\longrightarrow \mathcal{V}_{i})$. The intersection $\widetilde{S}$ of $R$ with the image  of the composition of projections $Q_{1}\rightarrow Q_{2}\rightarrow \cdots \rightarrow Q_{n-1}\rightarrow Q_{n}=\overline{R}$ is closed and $G$-invariant. Therefore the image of $\widetilde{S}$ in $M$, which is $S_{r_{1},\cdots, r_{n}}$, is closed.

Set $\Sigma=R\times_{\overline{R}}Q_{1}$ and $\mathbb{E}_{i}$ to the pullback of $\mathbb{E}'_{i}$ to $C\times \Sigma=(C\times Q_{i})\times_{Q_{i}} \Sigma$. Then it is easily seen that these have the required properties.
\end{proof}

Let $W$ be a vector bundle corresponding to a general point $[W]\in M$ of a component of $S_{r_{1},\cdots, r_{n}}$. Then $W$ has a filtration
$$0=W_{0}\subset W_{1}\subset W_{2}\subset\cdots \subset W_{n-1}\subset W_{n}=W$$
such that $V_{i}:=W_{i}/W_{i-1}$ is a vector bundle of rank $r_{i}-r_{i-1}$ and degree $d_{i}-d_{i-1}$ for any $1\leq i\leq n$. Let
$$a_{n}:=\text{dim }U_{C}(r_{n}-r_{n-1}, 1)+ \text{ext}^{1}(V_{n}, W_{n-1})-1,$$
and
 $$a_{i}:=\text{dim }U_{C}(r_{i}-r_{i-1}, 0)+ \text{ext}^{1}(V_{i}, W_{i-1})-1 \quad\text{ for any } \quad 1<i<n.$$
Then
\begin{equation}
 \text{dim }S_{r_{1},\cdots, r_{n}}\leq
\text{dim }U_{C}(r_{1}, 0)+a_{2}+\cdots +a_{n}-g.
\end{equation}

Since $\mu(V_{n})>\mu (W_{n-1})$, we have $\text{Hom}(V_{n}, W_{n-1})=0$
 and
$$\text{ext}^{1}(V_{n}, W_{n-1})=-\chi(V_{n}^{*}\otimes W_{n-1})=(r_{n}-r_{n-1})r_{n-1}(g-1)+r_{n-1}.$$
Hence
$$a_{n}=r_{n}(r_{n}-r_{n-1})(g-1)+r_{n-1}.$$

\begin{Lemma}
 If $V$ is a semi-stable vector bundle, then $h^{1}(V)\leq \text{rk}(V)\cdot g$ provided that $\text{rk}(V)+\text{deg}(V)\geq 0$.
\end{Lemma}

\begin{proof}
 If $\text{deg}(V)<0$, then $h^{0}(V)=0$ and by the Riemann-Roch , we have
$$h^{1}(V)=-\chi(V)=-(\text{deg}(V)+ \text{rk}(V)\cdot(1-g))=\text{rk}(V)\cdot g-(\text{rk}(V)+\text{deg}(V))\leq \text{rk}(V)\cdot g.$$
For $\text{deg}(V)\geq 0$, we can assume that the lemma holds for $V(-m)$ for some positive integer $m$ by induction. Therefore,
$$h^{1}(V)\leq h^{1}(V(-m))\leq\text{rk}(V(-m))\cdot g=\text{rk}(V)\cdot g$$
since $H^{1}(V/V(-m))=0$.
\end{proof}

From this lemma, for $2\leq i\leq n-1 $ we have
\begin{eqnarray*}
a_{i} &\leq&r_{i}(r_{i}-r_{i-1})(g-1)+r_{i-1}(r_{i}-r_{i-1}),
\end{eqnarray*}
and
\begin{eqnarray*}
\text{dim }U_{C}(r_{i-1}, 0)+a_{i}&\leq &r_{i-1}^{2}(g-1)+1+r_{i}(r_{i}-r_{i-1})(g-1)+r_{i-1}(r_{i}-r_{i-1})\\
&=& \text{dim }U_{C}(r_{i}, 0)-r_{i-1}(r_{i}-r_{i-1})(g-2).
\end{eqnarray*}
Hence
\begin{eqnarray*}
\text{dim }S_{r_{1},\cdots, r_{n}}&\leq& \text{dim}U_{C}(r_{1}, 0)+a_{2}+\cdots +a_{n}-g\\
&\leq & \cdots \leq \text{dim }U_{C}(r_{n-1},0)-\sum_{i=2}^{n-1}r_{i-1}(r_{i}-r_{i-1})(g-2)+a_{n}-g\\
&=&(r_{n}^{2}-1)(g-1)-(\sum_{i=2}^{n}r_{i-1}(r_{i}-r_{i-1})(g-2)+r_{n-1}(r_{n}-r_{n-1}-1)).
\end{eqnarray*}
Therefore
\begin{equation}
\text{Codim }S_{r_{1},\cdots, r_{n}}\geq \sum_{i=2}^{n}r_{i-1}(r_{i}-r_{i-1})(g-2)+r_{n-1}(r_{n}-r_{n-1}-1).
\label{eq:3}
\end{equation}

\begin{Theorem}
Any small rational curve in $M$ lies in a closed subset
$$S=\bigcup_{0<r_{1}<\cdots, <r_{n}=r} S_{r_{1},\cdots, r_{n}}$$
 of codimension at least
$$\text{min}_{0<r_{1}<\cdots <r_{n}=r}\{\sum_{i=2}^{n}r_{i-1}(r_{i}-r_{i-1})(g-2)+r_{n-1}(r_{n}-r_{n-1}-1)\}.$$
 where $0<r_{1}<\cdots <r_{n}=r$ runs over $n$ positive integers $r_{i}$ that satisfying $\sum_{i=1}^{n-1}r_{i}(\alpha_{i}-\alpha_{i+1})<r$ for some $n$ and some integers $\alpha_{1}>\dots>\alpha_{n-1}>\alpha_{n}$.
\end{Theorem}

\section{Small Rational Curves on $M=SU_{C}(3,\mathcal{L})$ with $\text{deg}\mathcal{L}=1$}
In this section, we want to determine all small rational curves on $M=SU_{C}(3, \mathcal{\mathcal{L}})$  with $\text{deg}\mathcal{L}=1$.

In this special case, Sun's degree formula (\ref{eq:2.1}) becomes
   $$\text{deg}(\phi^{*}(\Theta))=3\sum_{i=1}^{n}c_{2}(F'_{i})+
   \sum_{i=1}^{n-1}(\text{rk}E_{i}-3\text{deg}(E_{i}))(\alpha_{i}-\alpha_{i+1}).$$
For $M=SU_{C}(3,\mathcal{L})$, the degree of a small rational curve is either 1 or 2 with respect to $\Theta$. When $\text{deg}(\phi^{*}(\Theta))=1$, $\phi: \mathbb{P}^{1}\rightarrow M$ is a line, which has been studied in \cite{Sun05} and \cite{MokSun09}. Now we consider the case $\text{deg}(\phi^{*}(\Theta))=2$, and in the rest of this paper we say small rational curves only for this case.

 From section 2 and the degree formula (\ref{eq:1}) of a small rational curve
   $$2=\text{deg}(\phi^{*}(\Theta))=\sum_{i=1}^{n-1}\text{rk}E_{i}(\alpha_{i}-\alpha_{i+1}),$$
 any rational curve $\phi: \mathbb{P}^{1}\longrightarrow M$ is determined by a vector bundle $E$ on $X=C\times \mathbb{P}^{1}$ satisfying one of the following conditions:

  \textbf{(A)} The vector bundle $E$ fits in an exact sequence
   $$0\longrightarrow f^{*}V_{1}\bigotimes \pi^{*}\mathcal{O}_{\mathbb{P}^{1}}(2)\longrightarrow E\longrightarrow f^{*}V_{2}\longrightarrow 0,$$
   where $V_{1}, V_{2}$ are of rank 1, 2 and degrees 0, 1 respectively.

   \textbf{(B)} The vector bundle $E$ fits in an exact sequence
   $$0\longrightarrow f^{*}V_{1}\bigotimes \pi^{*}\mathcal{O}_{\mathbb{P}^{1}}(1)\longrightarrow E\longrightarrow f^{*}V_{2}\longrightarrow 0,$$
   where $V_{1}, V_{2}$ are of rank 2, 1 and degrees 0, 1 respectively.

\subsection{The constructions of small rational curves}
Let $U_{C}(2,1)$ be the moduli space of stable vector bundles of rank 2 and degree 1. It's known that $U_{C}(2,1)$ is a smooth projective variety and there is a universal vector bundle $\mathcal{V}$ on $C\times U_{C}(2,1)$. Let $J_{C}$ be the Jacobian of the curve $C$ and $\mathfrak{L}$ be a Poincare bundle on $C\times J_{C}$. Consider the morphism
\[\begin{CD}
J_{C}\times U_{C}(2,1)@>(\cdot)\times \text{det}(\cdot)>>J_{C}\times J_{C}^{1}
@>(\cdot)\otimes(\cdot)>>J_{C}^{1}
\end{CD} \]
and let $\mathcal{R}$ be its fiber at $[\mathcal{L}]\in J_{C}^{1}$. We still use $\mathcal{V}$ (resp. $\mathfrak{L}$) to denote the pullback on $C\times \mathcal{R}$ by the projection $C\times \mathcal{R}\longrightarrow C\times U_{C}(2,1)$ (resp. $C\times \mathcal{R}\longrightarrow C\times J_{C}$). Let $p:C\times \mathcal{R}\longrightarrow \mathcal{R}$ and $\mathcal{G}=R^{1}p_{*}(\mathcal{V}^{\vee}\bigotimes\mathfrak{L})$ which is a vector bundle of rank $2g-1$. Let $q: \mathcal{P}=\mathbb{P}(\mathcal{G})\longrightarrow \mathcal{R}$ be the projection bundle parametrizing 1-dimensional subspaces of $\mathcal{G}_{t}\ (t\in \mathcal{R})$ and $f: C\times \mathcal{P}\longrightarrow C$ and $\pi:C\times \mathcal{P}\longrightarrow \mathcal{P}$ be the projections. Then there is a universal extension
\begin{equation}
0\longrightarrow (Id_{C}\times q)^{*}\mathfrak{L}\bigotimes \pi^{*}\mathcal{O}_{\mathcal{P}}(1)\longrightarrow \mathcal{E}\longrightarrow (Id_{C}\times q)^{*}\mathcal{V}\longrightarrow 0
\label{eq:4}
\end{equation}
on $C\times \mathcal{P}$ such that for any point $x=([\xi],[V],[e])\in \mathcal{P}$, where $[\xi]\in J_{C},\ [V]\in U_{C}(2,1)$ with $\xi\bigotimes \text{det}(V)\simeq \mathcal{L}$ and $[e]\subset H^{1}(C, V^{\vee}\otimes \xi)$ being a line through the origin, the bundle $\mathcal{E}|_{C\times \{x\}}$ is the isomorphic class of vector bundles $E$ given by extensions
$$0\longrightarrow \xi \longrightarrow E \longrightarrow V\longrightarrow 0$$
that defined by vectors on the line $[e]\subset H^{1}(C, V^{\vee}\bigotimes \xi)$.

By [Lemma 3.1 in \cite{Sun05}], the vector bundle $\mathcal{E}$ given by the universal extension (\ref{eq:4}) on $C\times \mathcal{P}$ defines a morphism
$$\Phi: \mathcal{P}\longrightarrow SU_{C}(3,\mathcal{L})=M.$$

\begin{Lemma}
 Let $\alpha: \mathbb{P}^{1}\rightarrow \mathcal{P}$ be a morphism satisfying $\alpha^{*}\mathcal{O}_{\mathcal{P}}(1)\cong \mathcal{O}_{\mathbb{P}^{1}}(2)$, then $\alpha: \mathbb{P}^{1}\rightarrow \mathcal{P}$ is either a rational curve of degree 2 in $\mathcal{P}$ or a double cover of a line in $\mathcal{P}$.
\end{Lemma}

\begin{proof}
Let $Y:=\alpha(\mathbb{P}^{1})$ be the subvariety of $\mathcal{P}$ with the reduced structure and $\rho: \widetilde{Y}\rightarrow Y$ be its normalization. Then $\widetilde{Y}\cong \mathbb{P}^{1}$ and there is a morphism $\alpha':\mathbb{P}^{1}\rightarrow \widetilde{Y}$ such that $\alpha=\rho\circ\alpha'$. Thus
$$2=\text{deg}(\alpha^{*}\mathcal{O}_{\mathcal{P}}(1))=\text{deg}(\alpha')\cdot \text{deg}(\rho^{*}(\mathcal{O}_{\mathcal{P}}(1)|_{Y})),$$
and then we have
$$\text{deg}(\alpha')=1 \quad \text{and}\quad \text{deg}(\rho^{*}(\mathcal{O}_{\mathcal{P}}(1)|_{Y}))=2$$
or
$$\text{deg}(\alpha')=2 \quad \text{and}\quad \text{deg}(\rho^{*}(\mathcal{O}_{\mathcal{P}}(1)|_{Y}))=1.$$
The first case implies that $Y=\alpha(\mathbb{P}^{1})$ is a rational curve of degree 2 in $\mathcal{P}$ and $\alpha: \mathbb{P}^{1}\rightarrow \alpha(\mathbb{P}^{1})$ is the normalization of $\alpha(\mathbb{P}^{1})$. The second case implies that $Y=\alpha(\mathbb{P}^{1})$ is a line in $\mathcal{P}$ and $\alpha':\mathbb{P}^{1}\rightarrow \widetilde{Y}$ is a degree 2 morphism, i.e., $\alpha: \mathbb{P}^{1}\rightarrow \alpha(\mathbb{P}^{1})$ is a double cover of a line in $\mathcal{P}$.
\end{proof}

\begin{Remark}
 Let $\alpha: \mathbb{P}^{1}\rightarrow \mathcal{P}$ be a morphism satisfying $\alpha^{*}\mathcal{O}_{\mathcal{P}}(1)\cong \mathcal{O}_{\mathbb{P}^{1}}(p)$, where $p$ is a prime number. Then, similar to the proof of the above lemma, we can prove that $\alpha: \mathbb{P}^{1}\rightarrow \mathcal{P}$ is either a rational curve of degree $p$ in $\mathcal{P}$ or a \emph{$p$-fold cover of a line} in $\mathcal{P}$. In particular, when $p=3$, we call $\alpha: \mathbb{P}^{1}\rightarrow \mathcal{P}$ \emph{a triple cover of a line} in $\mathcal{P}$.
\end{Remark}

\begin{Proposition}
(1) For any rational curve of degree 2 in the fiber of $q: \mathcal{P}\longrightarrow \mathcal{R}$, its image is a small rational curve on $M$.

(2) For any double cover of a line in the fibre of $q: \mathcal{P}\longrightarrow \mathcal{R}$, its image is a small rational curve on $M$.
 \label{prop:3.4}
\end{Proposition}

\begin{proof}
Let $\alpha: \mathbb{P}^{1}\rightarrow \mathcal{P}$ be either a rational curve of degree 2 in the fibre of $q$ or a double cover of a line in the fibre of $q$, we always have $\alpha^{*}\mathcal{O}_{\mathcal{P}}(1)\cong \mathcal{O}_{\mathbb{P}^{1}}(2)$.
 Let $q(\alpha(\mathbb{P}^{1}))=([\xi],[V])\in \mathcal{R}$ and $E=(Id_{C}\times\alpha)^{*}\mathcal{E}$. Then the morphism
$$\Phi\circ \alpha: \mathbb{P}^{1}\longrightarrow M$$
is defined by $E$, which fits in an exact sequence
$$0\longrightarrow f^{*}\xi\bigotimes \pi^{*}\mathcal{O}_{\mathbb{P}^{1}}(2)\longrightarrow E \longrightarrow f^{*}V\longrightarrow 0.$$
Thus $c_{2}(E)=2$ and $c_{1}(E)^{2}=4$. Then the degree of $\text{deg}\Phi^{*}(\Theta)|_{\mathbb{P}^{1}}$ equals
\begin{eqnarray*}
 \frac{1}{2}\Delta(E)=3\times c_{2}(E)-\frac{1}{2}(3-1)c_{1}(E)^{2}=6-4=2,
\end{eqnarray*}
which implies that $\Phi(\alpha(\mathbb{P}^{1}))$ is a small rational curve in $M=SU_{C}(3,\mathcal{L})$ of degree 2 with respect to $\Theta$.
\end{proof}

\begin{Proposition}
 Let $\mathbb{P}^{1}\subseteq \mathcal{P}$ be a line, i.e., $\mathcal{O}_{\mathcal{P}}(1)|_{\mathbb{P}^{1}}\cong \mathcal{O}_{\mathbb{P}^{1}}(1)$, which is not in any fiber of $q$. Then
$\Phi|_{\mathbb{P}^{1}}: \mathbb{P}^{1}\longrightarrow M$ is a small rational curve if and only if $p_{2}\circ q(\mathbb{P}^{1})$ is a line in $SU_{C}(2,\mathcal{L'})$ for some line bundle $\mathcal{L'}$ of degree 1, where $p_{2}: \mathcal{R}\rightarrow U_{C}(2,1)$ is the projection.
\label{prop:3.5}
\end{Proposition}

\begin{proof}
 Let $p_{1}: \mathcal{R}\rightarrow J_{C}$ be the projection. Since $J_{C}$ is an abelian variety, $p_{1}\circ q: \mathbb{P}^{1}\rightarrow J_{C}$ is a constant morphism and let $p_{1}\circ q(\mathbb{P}^{1})=[\mathcal{L}_{1}]\in J_{C} $. Then $p_{2}\circ q|_{ \mathbb{P}^{1}}: \mathbb{P}^{1}\rightarrow U_{C}(2,1)$ factors through $SU_{C}(2, \mathcal{L}\otimes \mathcal{L}_{1}^{-1})$ and it is defined by a vector bundle $E'=(id_{C}\times p_{2}\circ q\circ \alpha)^{*}\mathcal{V}$ on $C\times \mathbb{P}^{1}$, where $\alpha: \mathbb{P}^{1}\rightarrow \mathcal{P}$ denotes the immersion. Thus $\Phi|_{\mathbb{P}^{1}}: \mathbb{P}^{1}\longrightarrow M$ is a rational curve defined by
$$(id_{C}\times \alpha)^{*}(0\longrightarrow (id_{C}\times q)^{*}\mathfrak{L}\bigotimes \pi^{*}\mathcal{O}_{\mathcal{P}}(1)\longrightarrow \mathcal{E}\longrightarrow (id_{C}\times q)^{*}\mathcal{V}\longrightarrow 0 ),$$
which is isomorphic to
\begin{equation}
0\longrightarrow f^{*}\mathcal{L}_{1}\bigotimes \pi^{*}\mathcal{O}_{\mathbb{P}^{1}}(1)\longrightarrow (id_{C}\times \alpha)^{*}\mathcal{E}:=E\longrightarrow E'\longrightarrow 0.
\label{eq:6}
\end{equation}

If $\Phi|_{\mathbb{P}^{1}}: \mathbb{P}^{1}\longrightarrow M$ is a small rational curve, we have known that $E$ satisfies either condition (A) or condition (B). If $E$ satisfy (A), then $\mathbb{P}^{1}\subseteq \mathcal{P}$ must be in a fiber of $q$. Thus $E$ satisfies (B), and then $E$ fits in an exact sequence
   $$0\longrightarrow f^{*}V_{1}\bigotimes \pi^{*}\mathcal{O}_{\mathbb{P}^{1}}(1)\longrightarrow E\longrightarrow f^{*}V_{2}\longrightarrow 0,$$
where $V_{1}, V_{2}$ are vector bundles of rank 2, 1 and degrees 0, 1 respectively.

Since $H^{0}(\mathbb{P}^{1}, \mathcal{O}_{\mathbb{P}^{1}}(-1))=0$, and by K\"{u}nneth formula, we have
$$H^{0}(C\times\mathbb{P}^{1}, f^{*}(\mathcal{L}_{1}^{-1}\otimes V_{2})\otimes \pi^{*}\mathcal{O}_{\mathbb{P}^{1}}(-1))=H^{0}(C, \mathcal{L}_{1}^{-1}\otimes V_{2})\otimes H^{0}(\mathbb{P}^{1}, \mathcal{O}_{\mathbb{P}^{1}}(-1)) = 0.$$
Thus
$$\text{Hom}(f^{*}\mathcal{L}_{1}\bigotimes \pi^{*}\mathcal{O}_{\mathbb{P}^{1}}(1), f^{*}V_{2}) =H^{0}(C\times\mathbb{P}^{1}, f^{*}(\mathcal{L}_{1}^{-1}\otimes V_{2})\otimes \pi^{*}\mathcal{O}_{\mathbb{P}^{1}}(-1))=0.$$
 Then there's an induced injective morphism $f^{*}\mathcal{L}_{1}\otimes \pi^{*}\mathcal{O}_{\mathbb{P}^{1}}(1)\rightarrow f^{*}V_{1}\otimes \pi^{*}\mathcal{O}_{\mathbb{P}^{1}}(1)$ and a morphism $\beta: E'\rightarrow f^{*}V_{2}$ satisfying a commutative diagram
\[\begin{CD}
0@>>>f^{*}\mathcal{L}_{1}\bigotimes \pi^{*}\mathcal{O}_{\mathbb{P}^{1}}(1)@>>>E@>>>E'@>>>0\\
@.@VVV@|@VV\beta V\\
0@>>>f^{*}V_{1}\bigotimes \pi^{*}\mathcal{O}_{\mathbb{P}^{1}}(1)@>>>E@>>>f^{*}V_{2}@>>>0.
\end{CD}\]
By the snake lemma, $\beta$ is surjective and $\text{Ker}(\beta)=f^{*}(V_{1}/\mathcal{L}_{1})\otimes\pi^{*}\mathcal{O}_{\mathbb{P}^{1}}(1)$. This implies that $p_{2}\circ q|_{ \mathbb{P}^{1}}: \mathbb{P}^{1}\rightarrow SU_{C}(2,\mathcal{L}\otimes \mathcal{L}_{1}^{-1})$ is a line in $SU_{C}(2,\mathcal{L}\otimes \mathcal{L}_{1}^{-1})$.

Conversely, if $p_{2}\circ q(\mathbb{P}^{1})$ is a line in $SU_{C}(2,\mathcal{L}\otimes \mathcal{L}_{1}^{-1})$, then by the result of lines in \cite{Sun05} and \cite{MokSun09}, there are line bundles $\mathcal{L}_{2}$, $\mathcal{L}_{3}$ on $C$, and an exact sequence
\begin{equation}
0\longrightarrow f^{*}\mathcal{L}_{2}\otimes \pi^{*}\mathcal{O}_{\mathbb{P}^{1}}(1)\longrightarrow E'\longrightarrow f^{*}\mathcal{L}_{3}\longrightarrow 0
\label{eq:7}
\end{equation}
on $C\times \mathbb{P}^{1}$ with $\mathcal{L}_{1}\otimes\mathcal{L}_{2}\otimes\mathcal{L}_{3}=\mathcal{L}$.
Then by (\ref{eq:6}) and (\ref{eq:7}), we have
 $c_{2}(E)=2,\ c_{1}(E)^{2}=4$ and $\Phi(\mathbb{P}^{1})$ is a small rational curve in $M=SU_{C}(3,\mathcal{L})$ of degree $2$ with respect to $\Theta$.
\end{proof}

Note that the a small rational curve in Proposition \ref{prop:3.4} is defined by a vector bundle satisfying (A), and a small rational curve in Proposition \ref{prop:3.5} is defined by a vector bundle satisfying (B) but $V_{1}$ is not stable. Now we will construct the small rational curves which are defined by vector bundles satisfying (B) and at the same time $V_{1}$ is stable.
Unlike the above way to obtain small rational curves, there is no universal  family of vector bundles on $C$ parametrised by $SU_{C}(2,0)$. Now we use the Hecke  transformation (\cite{NarasimhanRamanan75}, \cite{NarasimhanRamanan78}) to construct a family of stable vector bundles parametrised by a scheme which maps onto $SU_{C}(2,0)$.

Let $\mathcal{V}\longrightarrow C\times U_{C}(2,1)$ be a universal family of vector bundles of rank 2 and degree 1, and $\pi_{\mathcal{V}}: \mathbb{P}(\mathcal{V}^{\vee})\longrightarrow C\times  U_{C}(2,1)$ be the associated projective bundle. Then we can construct a family $K(\mathcal{V})$ of vector bundles of rank 2 and degree 0 on $C$ parametrised by $\mathbb{P}(\mathcal{V}^{\vee})$ fits in an exact sequence of sheaves on $C\times \mathbb{P}(\mathcal{V}^{\vee})$
$$0\longrightarrow K(\mathcal{V})^{\vee}\longrightarrow p_{\mathbb{P}(\mathcal{V}^{\vee})}^{*}\pi_{\mathcal{V}}^{*}\mathcal{V}\longrightarrow p_{\mathbb{P}(\mathcal{V}^{\vee})}^{*}\tau\bigotimes\mathcal{O}_{\mathbb{P}(\mathcal{V}^{\vee})}\longrightarrow 0,$$
where $\tau$ is the tautological bundle on $\mathbb{P}(\mathcal{V}^{\vee})$, and $\mathbb{P}(\mathcal{V}^{\vee})$ is considered as a divisor in $C\times \mathbb{P}(\mathcal{V}^{\vee})$ by the inclusion $(p_{C}\circ\pi_{\mathcal{V}}, Id_{\mathbb{P}(\mathcal{V}^{\vee})})$.

To define $K(\mathcal{V})$ we now construct a surjective homomorphism $\beta_{\mathcal{V}}:p_{\mathbb{P}(\mathcal{V}^{\vee})}^{*}\pi_{\mathcal{V}}^{*}\mathcal{V}\longrightarrow p_{\mathbb{P}(\mathcal{V}^{\vee})}^{*}\tau\otimes\mathcal{O}_{\mathbb{P}(\mathcal{V}^{\vee})}$
or, what is the same, an element of
$$H^{0}(C\times \mathbb{P}(\mathcal{V}^{\vee}), Hom(p_{\mathbb{P}(\mathcal{V}^{\vee})}^{*}\pi_{\mathcal{V}}^{*}\mathcal{V}, p_{\mathbb{P}(\mathcal{V}^{\vee})}^{*}\tau\otimes\mathcal{O}_{\mathbb{P}(\mathcal{V}^{\vee})})),$$
which is mapped isomorphically by $(p_{C}\circ\pi_{\mathcal{V}}, Id_{\mathbb{P}(\mathcal{V}^{\vee})})^{*}$ to
\begin{eqnarray*}
&&H^{0}(\mathbb{P}(\mathcal{V}^{\vee}), (p_{C}\circ\pi_{\mathcal{V}}, Id_{\mathbb{P}(\mathcal{V}^{\vee})})^{*} Hom(p_{\mathbb{P}(\mathcal{V}^{\vee})}^{*}\pi_{\mathcal{V}}^{*}\mathcal{V}, p_{\mathbb{P}(\mathcal{V}^{\vee})}^{*}\tau\otimes\mathcal{O}_{\mathbb{P}(\mathcal{V}^{\vee})}))\\
 &\approx& H^{0}(\mathbb{P}(\mathcal{V}^{\vee}), Hom((p_{C}\circ\pi_{\mathcal{V}}, Id_{\mathbb{P}(\mathcal{V}^{\vee})})^{*}p_{\mathbb{P}(\mathcal{V}^{\vee})}^{*}\pi_{\mathcal{V}}^{*}\mathcal{V}, (p_{C}\circ\pi_{\mathcal{V}}, Id_{\mathbb{P}(\mathcal{V}^{\vee})})^{*}(p_{\mathbb{P}(\mathcal{V}^{\vee})}^{*}\tau\otimes\mathcal{O}_{\mathbb{P}(\mathcal{V}^{\vee})})) )\\
 &\approx& H^{0}(\mathbb{P}(\mathcal{V}^{\vee}), Hom(\pi_{\mathcal{V}}^{*}\mathcal{V}, \tau)).
  \end{eqnarray*}

For the associated projective bundle $\pi_{\mathcal{V}}: \mathbb{P}(\mathcal{V}^{\vee})\longrightarrow C\times  U_{C}(2,1)$, there is a canonical surjective homomorphism $\alpha_{\mathcal{V}}: \pi_{\mathcal{V}}^{*}\mathcal{V}\longrightarrow \tau$. We define $\beta_{\mathcal{V}}$ by setting $(p_{C}\circ\pi_{\mathcal{V}}, Id_{\mathbb{P}(\mathcal{V}^{\vee})})^{*}\beta_{\mathcal{V}}=\alpha_{\mathcal{V}}$. It is easy to see that the homomorphism $\beta_{\mathcal{V}}$ thus defined is surjective, and the kernel of this homomorphism is locally free, which defines a vector bundle $H(\mathcal{V})$. Let $K(\mathcal{V})$ be its dual. Then $K(\mathcal{V})$ is a family of semi-stable vector bundles of rank 2 and degree 0 on $C$, and $K(\mathcal{V})$ induces a surjective morphism
$$\theta: \mathbb{P}(\mathcal{V}^{\vee})\longrightarrow U_{C}(2, 0).$$

Let $\mathbb{P}(\mathcal{V}^{\vee})^{s}:=\theta^{-1}(SU_{C}(2,0))$, where $SU_{C}(2,0)$ denotes the open set of stable bundles in $U_{C}(2,0)$. Then $K(\mathcal{V})|_{C\times \mathbb{P}(\mathcal{V}^{\vee})^{s}}$ is a family of stable vector bundles parametrised by $\mathbb{P}(\mathcal{V}^{\vee})^{s}$.

Now consider the morphism
\[\begin{CD}
g: SU_{C}(2,0)\times J_{C}^{1}@> \text{det}(\cdot)\times(\cdot)>>J_{C}\times J_{C}^{1}
@>(\cdot)\otimes(\cdot)>>J_{C}^{1}.
\end{CD} \]
Let $R'$ be its fiber of  at $[\mathcal{L}]\in J_{C}^{1}$ and $T$ be the fiber of $(\theta|_{\mathbb{P}(\mathcal{V}^{\vee})^{s}}\times Id_{J_{C}^{1}})\circ g$ at $[\mathcal{L}]\in J_{C}^{1}$, then $T=(\mathbb{P}(\mathcal{V}^{\vee})^{s}\times J_{C}^{1})\times_{SU_{C}(2,0)\times J_{C}^{-1}}R'$. We still use $K(\mathcal{V})$, $\mathfrak{L}^{1}$ to denote the pullback on $C\times T$ by the projections $C\times T\longrightarrow C\times \mathbb{P}(\mathcal{V}^{\vee})$ and $C\times T\longrightarrow C\times J_{C}^{1}$ respectively, where $\mathfrak{L}^{1}$ is a Poincare bundle on $C\times J_{C}^{1}$. Let $p: C\times T\longrightarrow T$ and $\mathcal{F}=R^{1}p_{*}(\mathfrak{L}^{1\vee}\otimes K(\mathcal{V}))$. Then $\mathcal{F}$ is a vector bundle of rank $2g$. Let $q': \mathcal{P}':=\mathbb{P}(\mathcal{F})\longrightarrow T$ be the projective bundle parametrizing 1-dimensional subspaces of $\mathcal{F}_{t}\ (t\in T)$ and $f: C\times \mathcal{P}'\longrightarrow  C$ and $\pi: C\times \mathcal{P}'\longrightarrow \mathcal{P}'$ be the projections. Then there is a universal extension
\begin{equation}
0\longrightarrow (Id_{C}\times q')^{*}K(\mathcal{V})\otimes \pi^{*}\mathcal{O}_{\mathcal{P}'}(1)\longrightarrow \mathcal{E'}\longrightarrow (Id_{C}\times q')^{*}\mathfrak{L}^{1}\longrightarrow 0
\label{eq:5}
\end{equation}
on $C\times \mathcal{P}'$ such that for any point $x=([0\rightarrow V_{1}^{\vee}\rightarrow \mathcal{V}\rightarrow \mathcal{O}_{p}\rightarrow 0], [\xi],[\iota])\in \mathcal{P}'$ with $\text{det}(V_{1})\otimes \xi =\mathcal{L}$, $0\rightarrow V_{1}^{\vee}\rightarrow \mathcal{V}\rightarrow \mathcal{O}_{p}\rightarrow 0$ represents a point in $\mathbb{P}(\mathcal{V}^{\vee})$ and $[\iota]\subset H^{1}(C, \xi^{\vee}\otimes V_{1})$ being a line through the origin, the bundle $\mathcal{E'}|_{C\times \{x\}}$ is the isomorphic class of vector bundle $E'$ given by extensions
$$0\longrightarrow V_{1}\longrightarrow E' \longrightarrow \xi \longrightarrow 0$$ that defined by vectors on the line $[\iota]\subset H^{1}(C, \xi^{\vee}\otimes V_{1})$.

\begin{Lemma}
Let $V_{1},V_{2}$ are vector bundles of ranks $r_{1}$, $r_{2}$ and degrees 0, 1 respectively. Let $0\longrightarrow V_{1}\longrightarrow V\longrightarrow V_{2}\longrightarrow 0$ be a non-trivial extension.

(i) If $V_{1}$ and $V_{2}$ are stable, then $V$ is stable;

(ii) if $V$ is stable,  then $V_{1}$  and $V_{2}$ are semistable.
\end{Lemma}

\begin{proof}
 (i) Let $V'\subset V$ be a proper subbundle and $V'_{2}\subset V_{2}$ be its image. Then we have a subbundle $V'_{1}\subset V_{1}$ such that $V'$ fits in an exact sequence
 $$0\longrightarrow V'_{1}\longrightarrow V'\longrightarrow V'_{2}\longrightarrow 0.$$

If $V'_{2}\neq V_{2}$, then $\text{deg}(V'_{2})\leq 0$, and we always have $\text{deg}(V'_{1})\leq0$. Hence $\mu (V')=\frac{\text{deg}(V'_{1})+\text{deg}(V'_{2})}{\text{rk}(V')}\leq 0<\mu(V)$. Thus we assume that $V'_{2}=V_{2}$, then $V'_{1}\neq V_{1}$ and $\text{deg}(V'_{1})<0$. Hence, $\mu(V')=\frac{\text{deg}(V'_{1})+1}{\text{rk}(V')}\leq 0<\mu(V)$.

 (ii) It's easy to check.
 \end{proof}

By the lemma above, $\mathcal{E'}$ is a family of stable vector bundles of rank 3 and with fixed determinant $\mathcal{L}$ on $C$ parametrised by $\mathcal{P}'$. Then $\mathcal{E'}$ defines a morphism
$$\Psi: \mathcal{P}':=\mathbb{P}(\mathcal{F})\longrightarrow  SU_{C}(3, \mathcal{L})=M.$$

\begin{Proposition}
 For any line $\mathbb{P}^{1}\subset \mathcal{P}'$ in the fiber of $q'$,  its image in $M$ is a small rational curve.
 \label{prop:3.7}
\end{Proposition}

 \begin{proof}
 Let $q'(\mathbb{P}^{1})=([0\rightarrow V_{1}^{\vee}\rightarrow \mathcal{V}\rightarrow \mathcal{O}_{p}\rightarrow 0], [\xi])\in T$ and $E'=\mathcal{E'}|_{C\times \mathbb{P}^{1}}$. Then the morphism
$$\Psi|_{\mathbb{P}^{1}}: \mathbb{P}^{1}\longrightarrow M$$
is defined by $E'$, which fits in an exact sequence
$$0\longrightarrow f^{*}V_{1}\otimes \pi^{*}\mathcal{O}_{\mathbb{P}^{1}}(1)\longrightarrow E'\longrightarrow f^{*}\xi\longrightarrow 0.$$
Thus $c_{2}(E')=2$ and $c_{1}(E')^{2}=4$. Then the degree of $\Psi|_{\mathbb{P}^{1}}$ is equal to
$$\text{deg}\Psi^{*}(\Theta)|_{\mathbb{P}^{1}}=\frac{1}{2}\bigtriangleup(E')=3c_{2}(E')-c_{1}(E')^{2}=2,$$
which implies that $\Psi(\mathbb{P}^{1})$ is a small rational curve in $M=SU_{C}(3,\mathcal{L})$ of $\Theta$-degree 2 and $\Psi|_{\mathbb{P}^{1}}$ is its normalization.
\end{proof}

\begin{Theorem}
There exist small rational curves on the moduli space $M$. Moreover, any rational curve obtained by the following ways are small rational curves and  any small rational curve $\phi: \mathbb{P}^{1}\longrightarrow M$  can be obtained by one of the following ways:

(i) it's the image (under $\Phi$) of a rational curve of degree 2 in $\mathcal{P}$ in the fiber of $q$;

(ii) it's the image (under $\Phi$) of a double cover of a line in $\mathcal{P}$ in the fiber of $q$;

(iii) it's the image (under $\Phi$) of a line in $\mathcal{P}$ which is not in a fiber of $q$ and maps to a line in $U_{C}(2,\mathcal{L'})$ for some line bundle $\mathcal{L'}$ of degree 1;

(iv) it's the image (under $\Psi$) of a line in $\mathcal{P}'$ in the fiber of $q'$.
\label{th:3.8}
\end{Theorem}

\begin{proof}
By Propositions \ref{prop:3.4}, \ref{prop:3.5} and \ref{prop:3.7}, rational curves obtained from (i), (ii), (iii) and (iv) are small rational curves.

If $\phi: \mathbb{P}^{1}\longrightarrow M$ is a small rational curve, then it's defined by a vector bundle $E$ on $C\times \mathbb{P}^{1}$ satisfied either condition (A) or condition (B). If $E$ satisfy (A), then $\phi: \mathbb{P}^{1}\longrightarrow M$ is the image (under $\Phi$) of either a rational curve of degree 2 in $\mathcal{P}$ in the fiber of $q$ at $([V_{1}], [V_{2}])\in \mathcal{R}$ or a double cover of a line in $\mathcal{P}$ in the fiber of $q$ at $([V_{1}], [V_{2}])\in \mathcal{R}$.

Now we assume that $E$ satisfies (B), i.e, small rational curve $\phi: \mathbb{P}^{1}\longrightarrow M$ is defined by a vector bundle $E$ on $C\times \mathbb{P}^{1}$, which fits in an exact sequence
   $$0\longrightarrow f^{*}V_{1}\bigotimes \pi^{*}\mathcal{O}_{\mathbb{P}^{1}}(1)\longrightarrow E\longrightarrow f^{*}V_{2}\longrightarrow 0,$$
where $V_{1}, V_{2}$ are vector bundles of ranks 2, 1 and degrees 0, 1 respectively.

If $V_{1}$ is semi-stable but not stable, then there is a line sub-bundle $\mathcal{L}_{1}$ of $V_{1}$ of degree 0, and then $f^{*}\mathcal{L}_{1}\otimes\pi^{*}\mathcal{O}_{\mathbb{P}^{1}}(1)$ is a sub-bundle of $E$. Hence $\phi: \mathbb{P}^{1}\longrightarrow M$ factors as $\mathbb{P}^{1}\rightarrow \mathcal{P}\rightarrow M$ where $\mathbb{P}^{1}\rightarrow \mathcal{P}$ is a rational curve of degree 1, and $p_{1}\circ q $ maps it to a point $[\mathcal{L}_{1}]\in J_{C}$. Let $E':=E/(f^{*}\mathcal{L}_{1}\otimes\pi^{*}\mathcal{O}_{\mathbb{P}^{1}}(1))$. Then $E'$ fits in an exact sequence
   $$0\longrightarrow f^{*}\mathcal{L}_{2}\otimes\pi^{*}\mathcal{O}_{\mathbb{P}^{1}}(1)\longrightarrow E'\longrightarrow f^{*}V_{2}\longrightarrow 0$$
   with $\mathcal{L}_{1}\otimes\mathcal{L}_{2}\otimes V_{2}=\mathcal{L}$.
    This implies that $p_{2}\circ q (\mathbb{P}^{1})$ is a line in $U_{C}(2, \mathcal{L}\otimes\mathcal{L}_{1}^{-1})$.

   If $V_{1}$ is stable, then  the small rational curve $\phi: \mathbb{P}^{1}\longrightarrow M$ factors through $\mathbb{P}\text{Ext}^{1}(V_{2}, V_{1})$. For any point $t=([0\rightarrow V_{1}^{\vee}\rightarrow V\rightarrow \mathcal{O}_{p}\rightarrow 0], [V_{2}])\in \theta^{-1}([V_{1}], [V_{2}])$, the fiber $q'^{-1}(t)$ of $q'$ over $t$ is isomorphic to $\mathbb{P}\text{Ext}^{1}(V_{2}, V_{1})$. Hence the small rational curve $\phi: \mathbb{P}^{1}\longrightarrow M$ can factors through the fiber $q'^{-1}(t)$, which means that it's the image (under $\Psi$) of a line in $\mathcal{P}'$ in the fiber of $q'$ at $t\in T$.
   \end{proof}

\begin{Remark}
  When a small rational curve $\phi: \mathbb{P}^{1}\longrightarrow M$ is defined by a vector bundle $E$ on $C\times  \mathbb{P}^{1}$ satisfied condition (B) and moreover $V_{1}$ is stable, then $\phi$ can be obtained by the fourth way in the above theorem, but the way is not unique. Indeed, it can factors through every fiber of $q'$ over $ \theta^{-1}([V_{1}],[V_{2}])$.
\end{Remark}

\subsection{Remarks on lines and the locus of small rational curves}
For the lines on the moduli space $M$, we have the following result:

 \begin{Theorem}
 There exist lines on the moduli space $M$. For any line $l\subseteq M$, there is a line $\mathbb{P}^{1}\subseteq \mathcal{P}$ in the fibre of $q$ such that $\Phi(\mathbb{P}^{1})=l$.
   \label{th:3.9}
 \end{Theorem}

 \begin{proof}
 It's a special case of Theorem 3.4 of \cite{Sun05} and Theorem 2.7 of \cite{MokSun09}.
 \end{proof}

 Recall that any small rational curve $\phi: \mathbb{P}^{1}\longrightarrow M$ is defined by a vector bundle $E$ on $C\times \mathbb{P}^{1}$ satisfied either condition (A) or condition (B), which implies that the image of $\phi$ lies in
  $S_{1,2} \text{ or } S_{2,1}$, which are defined as in section 2. And for any point $[V]\in S_{1,2}$, i.e., there is a sub line bundle $V_{1}\subseteq V$ of degree 0 and $V/V_{2}$ is a vector bundle. Let $\xi=([V_{1}], [V/V_{1}])\in \mathcal{R}$, and let $\mathcal{C}$ be a rational curve of degree 2 in $\mathcal{P}_{\xi}$ passing through $[0\rightarrow V_{1}\rightarrow V\rightarrow V/V_{1}\rightarrow 0]$, then the image of $\mathcal{C}$ in $M$ is a small rational curve passing through $[V]$. Similarly, for any point in $S_{2,1}$, there also exists a small rational curve passing through it. Hence

\begin{Proposition}
All the small rational curve lie in and cover $S_{1,2}\bigcup S_{2,1}$, which is a closed subset with codimension at least $2g-3$.
\end{Proposition}

 Similarly, we have

\begin{Proposition}
All the lines lie in and cover $S_{1,2}$, which is a closed subset with codimension at least $2g-3$.
\end{Proposition}

\subsection{ Remark on Hecke curves and their limits}
From \cite{Sun05}, we have known that, a minimal rational curve in $M=SU_{C}(3,\mathcal{L})$ passing through a generic point is a Hecke curve, which  has degree $3$ with respect to $\Theta$. In this subsection, we will study the limits of Hecke curves, i.e., the rational curves in $M$ has degree $3$ with respect to $\Theta$.

From Sun's degree formula (\ref{eq:2.1}), we have
$$3=\text{deg}\phi^{*}(\Theta)=3\sum_{i=1}^{n}c_{2}(F'_{i})+
   \sum_{i=1}^{n-1}(\text{rk}E_{i}-3\text{deg}(E_{i}))(\alpha_{i}-\alpha_{i+1}),$$
the vector bundle $E$ on $C\times \mathbb{P}^{1}$ which defines the rational curve $\phi: \mathbb{P}^{1}\longrightarrow M$ with degree 3, must satisfy one of the following cases:

\textbf{Case 1:} $n=1$ and $c_{2}(E)=1$;

\textbf{Case 2:} $\phi: \mathbb{P}^{1}\longrightarrow M$ is defined by a vector bundle $E$ on $C\times \mathbb{P}^{1}$ satisfying
$$0\longrightarrow f^{*}V_{1}\bigotimes \pi^{*}\mathcal{O}_{\mathbb{P}^{1}}(3)\longrightarrow E\longrightarrow f^{*}V_{2}\longrightarrow 0,$$
where $V_{1}, V_{2}$ are of rank 1, 2 and degrees 0, 1 respectively.

 \textbf{Case 3:} $\phi: \mathbb{P}^{1}\longrightarrow M$ is defined by a vector bundle $E$ on $C\times \mathbb{P}^{1}$ satisfying
 $$0\rightarrow f^{*}V_{1}\bigotimes \pi^{*}\mathcal{O}_{\mathbb{P}^{1}}(2)\longrightarrow E_{2}\longrightarrow f^{*}V_{2}\bigotimes \pi^{*}\mathcal{O}_{\mathbb{P}^{1}}(1)\rightarrow 0,$$
 $$0\rightarrow E_{2}\longrightarrow E\longrightarrow f^{*}V_{3}\rightarrow 0,$$
  where $V_{1}, V_{2}, V_{3}$ are line bundles of degrees 0, 0 and 1 respectively.

  If $E$ satisfy case 1, $\phi: \mathbb{P}^{1}\longrightarrow M$ is a Hecke curve. Similar as the discussion of small rational curves on $M$, we have

\begin{Theorem}
  If a rational curve of degree 3 (with respect to $\Theta$) is not a Hecke curve, it can be obtained from one of the following ways:

(i) it's the image (under $\Phi$) of a degree 3 rational curve in $\mathcal{P}$ in the fibre of $q$;

(ii) it's the image (under $\Phi$) of a triple cover of a line in $\mathcal{P}$ in the fibre of $q$;

(iii) it's the image (under $\Phi$) of a degree 2 rational curve in $\mathcal{P}$, which is not in a fibre of $q$ and maps to a line in $\mathcal{U}_{C}(2, \mathcal{L}')$ for some line bundle $\mathcal{L}'$ of degree 1;

(iv) it's the image (under $\Phi$) of a double cover of a line in $\mathcal{P}$, which is not in a fibre of $q$ and maps to a line in $\mathcal{U}_{C}(2, \mathcal{L}')$ for some line bundle $\mathcal{L}'$ of degree 1.

Moreover, any rational curve coming from one of above four ways is a rational curve of degree 3 (with respect to $\Theta$) which is not a Hecke curve.
\end{Theorem}

\section*{Acknowledgments}
The author would like to thank her supervisor Professor Xiaotao Sun for the helpful suggestions in the preparation of this paper.

\end{document}